\documentclass[11pt]{article}

\usepackage{amssymb,amsthm,amsmath,hyperref}

\numberwithin{equation}{section}

\newtheorem{thm}{Theorem}[section]
\newtheorem{lem}{Lemma}[section]
\newtheorem{cor}{Corollary}[section]
\newtheorem{prop}{Proposition}[section]
\theoremstyle{definition}
\newtheorem{defn}{Definition}[section]
\theoremstyle{remark}
\newtheorem{rem}{Remark}[section]

\allowdisplaybreaks

\begin{document}
\title{A note on spherically symmetric isentropic compressible flows with density-dependent viscosity coefficients\thanks{
This work is supported by NSFC 10571158  and China Postdoctoral
Science Foundation 20060400335} }
\author{Ting Zhang\thanks{E-mail: zhangting79@hotmail.com}, Daoyuan Fang\thanks{E-mail:
dyf@zju.edu.cn}\\
\textit{\small Department of Mathematics, Zhejiang University,
Hangzhou 310027, China} }
\date{}
\maketitle
\begin{abstract}
In this note, by constructing suitable approximate solutions, we
prove the existence of global weak solutions to the compressible
Navier-Stokes equations with density-dependent viscosity
coefficients in the whole space $\mathbb{R}^N$, $N\geq2$ (or
exterior domain), when the initial data are spherically symmetric.
In particular, we prove the existence of spherically symmetric
solutions to the Saint-Venant model for shallow water in the whole
space (or exterior domain). \\
\textbf{Keywords:} Compressible
Navier-Stokes equations; density-dependent viscosity coefficients
\end{abstract}
\section{Introduction}
In this note, we consider the following compressible Navier-Stokes
equations with density-dependent viscosity coefficients
    \begin{equation}
      \rho_t+\mathrm{div}(\rho \mathrm{U})=0,\label{SSIC-E1.1}
    \end{equation}
        \begin{equation}
          (\rho \mathrm{U})_t+\mathrm{div}(\rho \mathrm{U}\otimes \mathrm{U})-\mathrm{div}(2h(\rho)D(\mathrm{U}))
          -\nabla (g(\rho)\mathrm{div}\mathrm{U})+\nabla
          P(\rho)=0,\label{SSIC-E1.2}
        \end{equation}
where $t\in (0,+\infty)$ and $\mathrm{x}\in \mathbb{R}^N$,
$N\geq2$, $\rho(\mathrm{x},t)$, $\mathrm{U}(\mathrm{x},t)$ and
$P(\rho)=\rho^\gamma$ ($\gamma\geq1$) stand for the fluid density,
velocity and pressure respectively,
    $$
    D(\mathrm{U})=\frac{1}{2}(\nabla \mathrm{U}+(\nabla \mathrm{U})^\top)
    $$
is the strain tensor, $h(\rho)$ and $g(\rho)$ are the Lam\'{e}
viscosity coefficients satisfying
    \begin{equation}
      h(\rho)\geq0,\ 2h(\rho)+Ng(\rho)\geq0.
    \end{equation}

In the last several decades, significant progress on the system
(\ref{SSIC-E1.1})-(\ref{SSIC-E1.2}) with positive constant viscosity
coefficients has been achieved by many authors. In the case that the
initial data  are sufficiently regular and the initial density is
bounded away from zero, there exists a unique local strong solution,
and the solution exists globally in time provided that the initial
data are small perturbations of an uniform non-vacuum state. For
details, we refer the readers to papers
\cite{Danchin00,Matsumura1980} and the references therein.
 The situation becomes more complex in
the general case of nonnegative initial density, and a number of
important questions are still open. For example, the uniqueness of
global weak solutions. The first general result on the existence
of global weak solutions was obtained by Lions in
\cite{Lions1998}. There have been many generalizations of this
result, see
\cite{Feireisl2001,Hoff2004,Jiang2001,Jiang2003,Sun2006}. Using
the compatibility condition, Salvi-Stra\u{s}kraba \cite{Salvi93}
and Choe-Kim \cite{Choe03} obtained the existence and uniqueness
of the local strong solution.

The results in \cite{Hoff1991,Liu1998,Xin1998} show that the
compressible Navier-Stokes system with  constant viscosity
coefficients have the singularity in the presence of vacuum. By
some physical considerations,  Liu, Xin and Yang in \cite{Liu1998}
introduced the modified Navier-Stokes system with
density-dependent viscosity coefficients. As remarked in
\cite{Liu1998}, in the derivation of the Navier-Stokes equations
from the Boltzmann equation through the Chapman-Enskog expansion
to the second order, the viscosity is a function of the
temperature, and correspondingly depends on the density for
isentropic fluids. Meanwhile, in geophysical flows, many
mathematical models correspond to
(\ref{SSIC-E1.1})-(\ref{SSIC-E1.2}). In particular, the viscous
Saint-Venant system for shallow water is expressed exactly as
(\ref{SSIC-E1.1})-(\ref{SSIC-E1.2}) with $N=2$, $h(\rho)=\rho$,
$g(\rho)=0$ and $\gamma=2$
(\cite{Bresch2003-2,Bresch2003,Lions1998}).  As remarked in
\cite{Guo2007}, new mathematical challenges are encountered for
the shallow water equations and the multi-dimensional compressible
Navier-Stokes equations (\ref{SSIC-E1.1})-(\ref{SSIC-E1.2}). The
main difficulty is that the velocity can not be defined in the
vacuum state.

For one-dimensional compressible Navier-Stokes equations
(\ref{SSIC-E1.1})-(\ref{SSIC-E1.2}) with $h(\rho)=\rho^\theta$ and
$g(\rho)=0$, $\theta\in(0,1)$, there are many literatures on the
well-posedness theory of the solutions, see
\cite{fang,Jiang1998,Jiang2005,Liu1998,Yang2002,zhang2006-2}.
Considering the free boundary problem of the spherically symmetric
system, the local existence and uniqueness of the weak solution
were obtain in \cite{ChenP}, the large-time behavior of the global
solution for  data close to equilibrium was obtained in
\cite{Zhang2007,Zhang2008}. However, few results are available for
 multi-dimensional problems. In \cite{Bresch2003-2},  Bresch, Desjardins and Lin showed
 the existence of global weak solutions in dimension 2 or 3  for the Korteweg's system
 with the Korteweg stress tensor $k\rho\nabla\Delta\rho$. An
 interesting new entropy estimate is established in
 \cite{Bresch2003-2} in a priori way, which provided some high
 regularity for the density. Later, a similar result was obtained
 in \cite{Bresch2003} with an additional quadratic friction term
 $r\rho|\mathrm{U}|\mathrm{U}$. Recently, Mellet and Vasseur
 \cite{Mellet2007} proved the $L^1$ stability of weak solutions of
 the system (\ref{SSIC-E1.1})-(\ref{SSIC-E1.2}) with $N=2,3$ and $\gamma>1$, based on the new
 entropy estimate, extending the results in
 \cite{Bresch2003-2,Bresch2003} to the case $r=k=0$. Bresch and
 Desjardins constructed approximate solutions for the viscous
 shallow water system with drag terms or capillarity term and for
 the compressible Navier-Stokes equations with the cold pressure in
 \cite{Bresch2006}, and proved the global existence of weak
 solutions to these systems in \cite{Bresch2006,Bresch2007}. In
 \cite{Guo2007}, Guo, Jiu and Xin constructed a class of
 approximate solutions and proved the  existence of global weak
 solutions for the spherically symmetric  compressible
 Navier-Stokes equations with density-dependent viscosity in a
 bounded domain ($N=2,3$, $\gamma>1$).

In this note, we will construct a class of approximate solutions
and prove the global existence of weak solutions for the
spherically symmetric  compressible
 Navier-Stokes equations with density-dependent viscosity in the whole
 space  or
exterior domain ($N\geq2$, $\gamma\geq1$). Using the method in
\cite{Guo2007}, we can construct the
 approximate solutions on the annular domain
 $\{\varepsilon<|\mathrm{x}|<R\}$  by solving the approximate
 systems of (\ref{SSIC-E1.1})-(\ref{SSIC-E1.2}) with $h^\varepsilon(\rho)=h(\rho)+\varepsilon\rho^\theta$
  and
  $g^\varepsilon(\rho)=g(\rho)+(\theta-1)\varepsilon\rho^\theta$
  instead of $h(\rho)$ and $g(\rho)$. Then, using the usual zero
  extensions as in \cite{Hoff1992,Hoff2004}, we can construct the  approximate solutions
  on the  entire domain $\mathbb{R}^N$. But, the entropy estimates of  approximate
  solutions do not hold on the  entire domain $\mathbb{R}^N$, only hold on the annular domain.
  Using some techniques in Proposition \ref{SSIC-P3.3}, we can
  prove that $\nabla\sqrt{\rho}$ belongs to
  $L^\infty(0,T;L^2(\mathbb{R}^N))$, so that the nonlinear diffusion terms in the definition of weak solutions will make sense.
The extension method in \cite{Guo2007}, can preserves the uniform
$L^\infty(0,T;H^1(\Omega))$ estimate of $\sqrt{\rho^\varepsilon}$,
but seems not applicable to build approximate solutions   in the
whole
 space or
exterior domain.

\section{Statement of the results.}
The Cauchy problem of the compressible Navier-Stokes equations can
be written as
     \begin{equation}
      \rho_t+\mathrm{div}(\rho \mathrm{U})=0,\label{SSIC-E2.1}
    \end{equation}
        \begin{equation}
          (\rho \mathrm{U})_t+\mathrm{div}(\rho \mathrm{U}\otimes \mathrm{U})-\mathrm{div}(2h(\rho)D(\mathrm{U}))
          -\nabla (g(\rho)\mathrm{div}\mathrm{U})+\nabla P(\rho)=0,\label{SSIC-E2.2}
        \end{equation}
with initial conditions
    \begin{equation}
    \rho|_{t=0}=\rho_0\geq0,\ \rho \mathrm{U}|_{t=0}=\mathrm{m}_0.\label{SSIC-E2.3}
    \end{equation}
Before introducing the notion of weak solution, let us state the
assumptions on the viscosity coefficients, as in
\cite{Mellet2007}.

\medskip \noindent\textbf{Conditions on $h(\rho)$ and $g(\rho)$:}

We assume that $h(\rho)$ and $g(\rho)$ are two $C^2(0,\infty)$
functions satisfying
    \begin{equation}
      g(\rho)=2\rho h'(\rho)-2h(\rho),\label{SSIC-E2.4}
    \end{equation}
    \begin{equation}
      h'(\rho)\geq\nu,\ h(0)\geq0,\label{SSIC-E2.5}
    \end{equation}
        \begin{equation}
          |g'(\rho)|\leq \frac{1}{\nu}h'(\rho),
        \end{equation}
            \begin{equation}
              \nu_1 h(\rho)\leq 2h(\rho)+Ng(\rho)\leq\label{SSIC-E2.7}
              \nu_2h(\rho),
            \end{equation}
    where $\nu\in(0,1)$ and $\nu_2\geq \nu_1>0$ are three constants satisfying
    \begin{equation}
      \frac{4N-4\sqrt{2N^2-4N+4}}{N^2-4N+4}<\frac{\nu_1-2}{N},
           \ \frac{4N+4\sqrt{2N^2-4N+4}}{N^2-4N+4}>\frac{\nu_2-2}{N},
           \ N\geq3.\label{SSIC-E2.8-1}
    \end{equation}
     When $N\geq3$ and
    $\gamma\geq\frac{N}{N-2}$, we also require that
        \begin{equation}
          \liminf_{\rho\rightarrow\infty}\frac{h(\rho)}{\rho^{\frac{N-2}{N}\gamma+\varepsilon}}>0,\label{SSIC-E2.8}
        \end{equation}
    for some small $\varepsilon>0$.

\begin{rem}
  From the above conditions, one has
    \begin{equation}
      \left\{
      \begin{array}{ll}
        C\rho^{\frac{N-1}{N}+\frac{\nu_1}{2N}}\leq h(\rho)\leq  C\rho^{\frac{N-1}{N}+\frac{\nu_2}{2N}}, &\rho\geq1,\\
    C\rho^{\frac{N-1}{N}+\frac{\nu_2}{2N}}\leq h(\rho)\leq  C\rho^{\frac{N-1}{N}+\frac{\nu_1}{2N}}, &\rho\leq1.
      \end{array}
      \right.
    \end{equation}
\end{rem}

\begin{defn}\label{SSIC-D2.1}
  We say that $(\rho,\mathrm{U})$ is a weak solution of
  (\ref{SSIC-E2.1})-(\ref{SSIC-E2.3}) on $\mathbb{R}^N\times[0,T]$,
  provided that

(1)
    $$
    \rho\in L^\infty(0,T;L^1\cap L^\gamma(\mathbb{R}^N)),
    \    \sqrt{\rho}\in L^\infty(0,T;H^1(\mathbb{R}^N)),
    $$
        $$
    \sqrt{\rho}\mathrm{U}\in L^\infty(0,T;(L^2(\mathbb{R}^N))^N),
        $$
            $$
    h(\rho)D(\mathrm{U}) \in
L^2(0,T;(W^{-1,1}_{\mathrm{loc}}(\mathbb{R}^N))^{N\times N}),
 \ g(\rho)\mathrm{div} \mathrm{U}\in
L^2(0,T;W^{-1,1}_{\mathrm{loc}}(\mathbb{R}^N)),
            $$
with $\rho\geq0$;

(2) For any $t_2> t_1\geq0$ and $\phi_1\in
C^1_c(\mathbb{R}^N\times[0,\infty))$, the mass equation
(\ref{SSIC-E2.1}) holds in the following sense:
    \begin{equation}
    \int_{\mathbb{R}^N}\rho\phi_1 d\mathrm{x}|^{t_2}_{t_1}=\int^{t_2}_{t_1}\int_{\mathbb{R}^N}
    (\rho\partial_t\phi_1+\rho\mathrm{U}\cdot\nabla\phi_1)d\mathrm{x}dt;\label{SSIC-E2.10}
    \end{equation}

(3) The following equality holds for all  smooth test function
$\phi_2(t,x)\in (C^2_c(\mathbb{R}^N\times[0,\infty)))^N$ with
$\phi_2(T,\cdot)=0$:
    \begin{eqnarray}
      &&\int_{\mathbb{R}^N}\mathrm{m}_0\cdot\phi_2(0,x)d\mathrm{x}+\int^T_0\int_{\mathbb{R}^N}
      \left(\sqrt{\rho}(\sqrt{\rho} \mathrm{U})\cdot\partial_t\phi_2+
      \sqrt{\rho}\mathrm{U}\otimes\sqrt{\rho}\mathrm{U}:\nabla\phi_2
      \right)d\mathrm{x}dt\nonumber\\
            &&+\int^T_0\int_{\mathbb{R}^N}\rho^\gamma\mathrm{div}\phi_2
            d\mathrm{x}dt
            -<2h(\rho)D(\mathrm{U}) ,\nabla \phi_2>-<g(\rho)\mathrm{div} \mathrm{U},\mathrm{div}
            \phi_2>=0,\label{SSIC-E2.11}
    \end{eqnarray}
where the diffusion terms make sense when written as
    \begin{eqnarray*}
    <2h(\rho)D(\mathrm{U}),\nabla
    \phi>&=&-\int_{\mathbb{R}^N}\frac{h(\rho)}{\sqrt{\rho}}(\sqrt{\rho}\mathrm{U}_j)
    \partial_{ii}\phi_jd\mathrm{x}dt-
    \int_{\mathbb{R}^N}(\sqrt{\rho}\mathrm{U}_j)2h'(\rho)\partial_i\sqrt{\rho}\partial_i\phi_jd\mathrm{x}dt\\
    &&    -\int_{\mathbb{R}^N}\frac{h(\rho)}{\sqrt{\rho}}(\sqrt{\rho}\mathrm{U}_i)
    \partial_{ji}\phi_jd\mathrm{x}dt-
    \int_{\mathbb{R}^N}(\sqrt{\rho}\mathrm{U}_i)2h'(\rho)\partial_j\sqrt{\rho}\partial_i\phi_jd\mathrm{x}dt,
    \end{eqnarray*}
and
    $$
    <g(\rho)\mathrm{div} \mathrm{U},\mathrm{div}
            \phi>=-\int_{\mathbb{R}^N}\frac{g(\rho)}{\sqrt{\rho}}(\sqrt{\rho}\mathrm{U}_j)
    \partial_{ij}\phi_id\mathrm{x}dt-
    \int_{\mathbb{R}^N}(\sqrt{\rho}\mathrm{U}_j)2g'(\rho)\partial_j\sqrt{\rho}\partial_i\phi_id\mathrm{x}dt.
    $$
\end{defn}

In this paper, we will construct global spherically symmetric weak
solutions to (\ref{SSIC-E2.1})-(\ref{SSIC-E2.3}). The initial data
are assumed  to satisfy
    \begin{equation}
      \rho_0\geq0\ \textrm{a.e. in}\ \mathbb{R}^N,
      \ \mathrm{m}_0=0\ \textrm{a.e. on}\
      \{x\in\mathbb{R}^N|\rho_0(x)=0\},\label{SSIC-E2.12}
    \end{equation}
        \begin{equation}
          \rho_0\in L^1\cap L^\gamma(\mathbb{R}^N),
          \ \frac{\nabla   h(\rho_0)}{\sqrt{\rho_0}}\in L^2(\mathbb{R}^N),
          \ \frac{\mathrm{m}^2_0}{\rho_0}(1+\ln(1+\frac{\mathrm{m}^2_0}{\rho_0^2}))\in L^1(\mathbb{R}^N).\label{SSIC-E2.13}
        \end{equation}

 The main result of
this paper is the following:
\begin{thm}\label{SSIC-T2.1}
  Assume that $\gamma\geq1$, $h(\rho)$ and $g(\rho)$  satisfy conditions
  (\ref{SSIC-E2.4})-(\ref{SSIC-E2.8}). If the initial data have
  the form
    $$
    \rho_0=\rho_0(|\mathrm{x}|),
    \ \mathrm{m}_0=m_0(|\mathrm{x}|)\frac{\mathrm{x}}{r}
    $$
  and satisfy (\ref{SSIC-E2.12})-(\ref{SSIC-E2.13}), then the
  initial-value problem (\ref{SSIC-E2.1})-(\ref{SSIC-E2.3}) has a
  global spherically symmetric weak solution
        $$
         \rho=\rho(|\mathrm{x}|,t),
    \ \mathrm{U}=u(|\mathrm{x}|,t)\frac{\mathrm{x}}{r}
    $$
  satisfying for all $T>0$,
    \begin{equation}
      \rho(\mathrm{x},t)\in C([0,T];L^1(\mathbb{R}^N))),
     \label{SSIC-E2.14}
    \end{equation}
        \begin{equation}
          \int_{\mathbb{R}^N}\rho(\mathrm{x},t)d\mathrm{x}=
                    \int_{\mathbb{R}^N}\rho_0(\mathrm{x})d\mathrm{x}.\label{SSIC-E2.15}
        \end{equation}
  Moreover, it holds that
        \begin{equation}
          \sup_{t\in[0,T]}
           \int_{\mathbb{R}^N}
          \left(
                   \frac{1}{\rho}|\nabla
          h(\rho)|^2+\rho(1+|\mathrm{U}|^{2})(1+\ln(1+|\mathrm{U}|^{2}))
          \right)d\mathrm{x}\leq C,\label{SSIC-E2.16}
        \end{equation}
  where $C$ is a constant.
\end{thm}
\begin{rem}
  Using the similar argument as that in \cite{Guo2007}, one can obtain
  that
     $$
          \sup_{t\in[0,T]}
           \int_{\mathbb{R}^N}
         \rho|\mathrm{U}|^{2+\eta}
         d\mathrm{x}\leq C,
        $$
 when $
           \int_{\mathbb{R}^N}
         \rho_0|\mathrm{U_0}|^{2+\eta}
         d\mathrm{x}\leq C
        $ for some small $\eta\in(0,1)$.
\end{rem}
\begin{rem}
Similarly,  using the usual zero
  extension method, one can obtain the similar result for the  existence of
global weak
 solutions for the spherically symmetric  compressible
 Navier-Stokes equations with density-dependent viscosity in a
 bounded domain ($N\geq2$, $\gamma\geq1$).
\end{rem}
\begin{rem}
  Under conditions   (\ref{SSIC-E2.4})-(\ref{SSIC-E2.8}), using
  the similar argument as that in \cite{Mellet2007}, one can easily obtain
  the similar result as that in \cite{Mellet2007} with
  $N\geq2$ and $\gamma\geq1$.
\end{rem}

\vspace{0.3cm}\noindent$\bullet$ \textbf{Exterior problem}
\vspace{0.3cm}

 Using the similar proof of Theorem \ref{SSIC-T2.1}, we can study
 the following exterior problem:
 \begin{equation}
      \rho_t+\mathrm{div}(\rho \mathrm{U})=0,
      \ t>0,\mathrm{x}\in\Omega,\label{SSIC-E2.1-E}
    \end{equation}
        \begin{equation}
          (\rho \mathrm{U})_t+\mathrm{div}(\rho \mathrm{U}\otimes \mathrm{U})-\mathrm{div}(2h(\rho)D(\mathrm{U}))
          -\nabla (g(\rho)\mathrm{div}\mathrm{U})+\nabla P(\rho)=0,\label{SSIC-E2.2-E}
        \end{equation}
with boundary and initial conditions
    \begin{equation}
   (\rho\mathrm{U})|_{\mathrm{x}\in\partial \Omega}=0,
   \ \rho|_{t=0}=\rho_0\geq0,\ \rho \mathrm{U}|_{t=0}=\mathrm{m}_0,\label{SSIC-E2.3-E}
    \end{equation}
where $\Omega=\{\mathrm{x}\in\mathbb{R}^N||x|>1\}$, $N\geq2$.

\begin{defn}\label{SSIC-D2.2}
  We say that $(\rho,\mathrm{U})$ is a weak solution of
  (\ref{SSIC-E2.1-E})-(\ref{SSIC-E2.3-E}) on $\Omega\times[0,T]$,
  provided

(a) The condition (1)
   in Definition \ref{SSIC-D2.1}  where $\mathbb{R}^N$ is replaced by $\Omega$;

 (b) For any $t_2> t_1\geq0$ and $\phi_1\in
C^1_c(\mathbb{R}^N\times[0,\infty))$, the mass equation
(\ref{SSIC-E2.1}) holds in the following sense:
    \begin{equation}
    \int_{\Omega}\rho\phi_1 d\mathrm{x}|^{t_2}_{t_1}=\int^{t_2}_{t_1}\int_{\Omega}
    (\rho\partial_t\phi_1+\rho\mathrm{U}\cdot\nabla\phi_1)d\mathrm{x}dt;\label{SSIC-E2.10-E}
   \end{equation}

(c) The condition (3)
   in Definition \ref{SSIC-D2.1}  where $\mathbb{R}^N$ is replaced by $\Omega$.
\end{defn}

 Using the similar proof of Theorem \ref{SSIC-T2.1} and $\|\sqrt{\rho(r)}\|_{L^\infty([1,\infty))}\lesssim
  \|\sqrt{\rho(r)}\|_{H^1([1,\infty))}$, we can obtain the similar result without the condition (\ref{SSIC-E2.8}). Here, we give
  the following theorem and omit the  proof.
\begin{thm}\label{SSIC-T2.2}
  Assume that $\gamma\geq1$, $h(\rho)$ and $g(\rho)$   satisfy conditions
  (\ref{SSIC-E2.4})-(\ref{SSIC-E2.7}). If the initial data have
  the form
    $$
    \rho_0=\rho_0(|\mathrm{x}|),
    \ \mathrm{m}_0=m_0(|\mathrm{x}|)\frac{\mathrm{x}}{r}
    $$
  and satisfy (\ref{SSIC-E2.12})-(\ref{SSIC-E2.13}) where $\mathbb{R}^N$ is replaced by $\Omega$, then the
  initial-value problem (\ref{SSIC-E2.1})-(\ref{SSIC-E2.3}) has a
  global spherically symmetric weak solution
        $$
         \rho=\rho(|\mathrm{x}|,t),
    \ \mathrm{U}=u(|\mathrm{x}|,t)\frac{\mathrm{x}}{r}
    $$
  satisfying (\ref{SSIC-E2.14})-(\ref{SSIC-E2.16})  where $\mathbb{R}^N$ is replaced by $\Omega$, for all $T>0$.
\end{thm}

\begin{rem}
   In particular, we get the
existence of spherically symmetric solutions to the Saint-Venant
model for shallow water system in the whole space or  exterior
domain.
\end{rem}

\section{Proof of Theorem \ref{SSIC-T2.1}}

The key point of the proof of Theorem \ref{SSIC-T2.1} is to
construct smooth approximate solutions satisfying the a priori
estimates required in the $L^1$ stability analysis. The crucial
issue is to obtain lower and upper bounds of the density. To this
end, we study the following system as an approximate system of
 (\ref{SSIC-E2.1})-(\ref{SSIC-E2.2}).
    \begin{equation}
      \rho_t+\mathrm{div}(\rho \mathrm{U})=0,\label{SSIC-E3.1}
    \end{equation}
        \begin{eqnarray}
          &&(\rho \mathrm{U})_t+\mathrm{div}(\rho \mathrm{U}\otimes \mathrm{U})
          -\mathrm{div}((2h(\rho)+\varepsilon\rho^{\theta})D(\mathrm{U}))\nonumber\\
                &&          -\nabla ((g(\rho)+(\theta-1)\varepsilon\rho^{\theta})\mathrm{div}\mathrm{U})+\nabla P(\rho)=0,\label{SSIC-E3.2}
        \end{eqnarray}
where $\varepsilon>0$ is a constant and
$\theta=\frac{N-1+\alpha}{N}$ with $\alpha\in(0,1)$ satisfying
    \begin{equation}
   V_1(\frac{N}{1-\alpha})<\min\{\frac{\nu_1-2}{N},\frac{\alpha-1}{N}\},
   \ V_2(\frac{N}{1-\alpha})>\frac{\nu_2-2}{N},\label{SSIC-E3.3-1}
    \end{equation}
where
    $$
    V_1(m)=\frac{4N(m-1)-4\sqrt{N^2(m-1)^2+(N-1)(m-1)(m-2)^2}}{(N-1)(m-2)^2}
    $$
and
    $$
    V_2(m)=\frac{4N(m-1)+4\sqrt{N^2(m-1)^2+(N-1)(m-1)(m-2)^2}}{(N-1)(m-2)^2}.
    $$

\begin{rem}
  From (\ref{SSIC-E2.8-1}), we can choose a small constant
  $\alpha$ satisfying (\ref{SSIC-E3.3-1}).
\end{rem}

When $\rho(\mathrm{x},t)=\rho(r,t)$,
$\mathrm{U}(\mathrm{x},t)=u(r,t)\frac{\mathrm{x}}{r}$,
(\ref{SSIC-E3.1})-(\ref{SSIC-E3.2}) becomes
    \begin{equation}
      \rho_t+(\rho u)_r+\frac{(N-1)\rho u}{r}=0,\label{SSIC-E3.3}
    \end{equation}
        \begin{equation}
          \rho u_t+\rho u u_r+(\rho^\gamma)_r+
          (2h +\varepsilon\rho^\theta)_r\frac{(N-1)u}{r}
          =((2h+g+\theta\varepsilon\rho^\theta)(u_r+\frac{(N-1)u}{r}))_r,\label{SSIC-E3.4}
        \end{equation}
for $r>0$. We will first construct the smooth solution of
(\ref{SSIC-E3.3})-(\ref{SSIC-E3.4}) in the truncated region
$0<\varepsilon<r<R<\infty$ with the following boundary conditions
and initial condition
    \begin{equation}
      u(r,t)|_{r=\varepsilon}=
            u(r,t)|_{r=R}=0,\label{SSIC-E3.6-10}
    \end{equation}
        \begin{equation}
          (\rho, u)(r,0)=(\rho_{0,\varepsilon,R,\delta},u_{0,\varepsilon,R,\delta}):=(\rho_{0,\varepsilon,R}*J_\delta,
          u_{0,\varepsilon,R}*J_\delta),\ \varepsilon<r<R,\label{SSIC-E3.6}
        \end{equation}
where $J_\delta$ is a standard mollifier,
    $$
    \rho_{0,\varepsilon,R}(r)=\left\{
    \begin{array}{ll}
    \rho_0(\varepsilon)+\varepsilon,&r\in[0,\varepsilon],\\
        \rho_0(r)+\varepsilon,&r\in[\varepsilon,R],\\
            \rho_0(R)+\varepsilon,&r\in[R,\infty),
    \end{array}
    \right.
    $$
and
    $$
    u_{0,\varepsilon,R}(r)=\left\{
    \begin{array}{ll}
    0,&r\in[0,\varepsilon+2\delta],\\
        \frac{m_0(r)}{\rho_0(r)+\varepsilon},&r\in[\varepsilon+2\delta,R-2\delta],\\
            0,&r\in[R-2\delta,\infty).
    \end{array}
    \right.
    $$

We assume that $\varepsilon$ and $R$ satisfy $\varepsilon
R^N\leq\sqrt{\varepsilon}$. Letting $\varepsilon\rightarrow0$ and
$R\rightarrow\infty$, we can easily obtain that
$(\rho_{0,\varepsilon,R}, u_{0,\varepsilon,R})$ convergence to
$(\rho_{0}, u_{0})$ in  spaces given in (\ref{SSIC-E2.13}). From
(\ref{SSIC-E3.3-1}) and similar arguments as that in
\cite{Guo2007}, one can obtain the smooth solutions
$(\rho^{\varepsilon,R,\delta}(r,t),u^{\varepsilon,R,\delta}(r,t))$
to the approximate system (\ref{SSIC-E3.3})-(\ref{SSIC-E3.6}).

\begin{rem}
  To obtain the existence of
  $(\rho^{\varepsilon,R,\delta}(r,t),u^{\varepsilon,R,\delta}(r,t))$,
  we need to consider the following system in the Lagrangian
  coordinates:
    $$
    \rho_\tau+\rho^2(r^{N-1}u)_x=0,
    $$
        $$
    r^{1-N}u_\tau+(\rho^\gamma)_x=[(\rho h+\rho g+\varepsilon\theta\rho^{\theta+1})
    (r^{N-1}u)_x]_x-(h+\varepsilon\rho^\theta)_x\frac{(N-1)u}{r},
        $$
    \begin{equation*}
      u(0,\tau)=
            u(1,\tau)=0,
    \end{equation*}
        \begin{equation*}
          (\rho,
          u)(\cdot,0)=(\rho_{0,\varepsilon,R,\delta},u_{0,\varepsilon,R,\delta}).
        \end{equation*}
From (\ref{SSIC-E3.3-1}) and similar arguments as that in
\cite{Guo2007}, one can obtain that $u\in
L^\infty(0,T;L^{\frac{N}{1-\alpha}}_x)$, $(\rho^\theta)_x\in
L^\infty(0,T;L^{\frac{N}{1-\alpha}}_x)$ and $\rho^{-1}\in
L^\infty_{\tau x}$ (for simplicity, we omit the superscript). To
estimate $\|u\|_{L^\infty(0,T;L^{\frac{N}{1-\alpha}}_x)}$, we need
to estimate the following terms
    \begin{eqnarray}
    &&-(m-1)(2h\rho+g\rho+\varepsilon\theta\rho^{\theta+1})r^{2N-2}u^{m-2}u_x^2\nonumber\\
        &&  -[2(N-1)\rho h+(N-1)^2\rho
    g+\varepsilon(N-1)(\theta(N-1)-N+2)\rho^{1+\theta}]r^{-2}\rho^{-2}u^m\nonumber\\
        &&-[\rho g m(N-1)+\varepsilon
        m(N-1)(\theta-1)\rho^{1+\theta}]r^{N-2}\rho^{-1}u^{m-1}u_x,\label{SSIC-E3.8-1}
    \end{eqnarray}
where $m=\frac{N}{1-\alpha}$. From (\ref{SSIC-E3.3-1}), we have
    $$(\ref{SSIC-E3.8-1})\leq -C(\rho h+\varepsilon\rho^{\theta+1})(r^{2N-2}u^{m-2}u_x^2+r^{-2}\rho^{-2}u^m).$$
Then, using similar arguments as that in \cite{Guo2007}, one can
obtain  $u\in L^\infty(0,T;L^{\frac{N}{1-\alpha}}_x)$.
\end{rem}

So far,
$$(\rho^{\varepsilon,R,\delta},u^{\varepsilon,R,\delta})$$ are
defined on $\varepsilon\leq r\leq R$. To take the limit
$(\varepsilon_j,R_j,\delta_j)\rightarrow(0,\infty,0)$, we extend
$(\rho^{\varepsilon,R,\delta},u^{\varepsilon,R,\delta})$ to the
whole space $\mathbb{R}^N$ in the following way
    \begin{equation}
      \widetilde{\rho}^{\varepsilon,R,\delta}(r,t)
      =\left\{\begin{array}{ll}
    \rho^{\varepsilon,R,\delta}(r,t),&r\in[\varepsilon,R],\\
        0,&\mathrm{else},
              \end{array}
      \right.\label{SSIC-E3.7}
    \end{equation}
    \begin{equation}
      \widetilde{u}^{\varepsilon,R,\delta}(r,t)
      =\left\{\begin{array}{ll}
    u^{\varepsilon,R,\delta}(r,t),&r\in[\varepsilon,R],\\
        0,&\mathrm{else}.
              \end{array}
      \right.\label{SSIC-E3.8}
    \end{equation}
For simplicity, we denote the obtained  approximate solutions
$\{\widetilde{\rho}^{\varepsilon_j,R_j,\delta_j},\widetilde{u}^{\varepsilon_j,R_j,\delta_j}\}$
by $\{\rho^j,{u}^j\}$. Let $\rho^j(\mathrm{x},t)=\rho^j(r,t)$,
$\mathrm{U}^j(\mathrm{x},t)=u^j(r,t)\frac{\mathrm{x}}{r}$,
$B_{\varepsilon,R}=\{\mathrm{x}\in \mathbb{R}^N|\
\varepsilon<|\mathrm{x}|<R\}$ and $B_{R}=\{\mathrm{x}\in
\mathbb{R}^N|\ |\mathrm{x}|<R\}$.

Using similar arguments as that in  proofs of Lemmas 3.2 and 4.1
in \cite{Guo2007}, and the similar argument as that in
\cite{Lions1998} (\S5.5),  we have the following lemma.
\begin{lem}\label{SSIC-L3.1}
  There exists a constant $C$ independent of $\varepsilon$, $R$ and
  $\delta$ such that
        \begin{equation}
        \sup_{t\in[0,T]}\int_{\mathbb{R}^N}\rho^j(\mathrm{x},t)d\mathrm{x}\leq
        C,\label{SSIC-E3.9}
    \end{equation}
            \begin{eqnarray}
        &&\sup_{t\in[0,T]}\int_{\mathbb{R}^N}\left(
        \frac{1}{2}\rho^j|\mathrm{U}^j|^2+\frac{1}{\gamma-1}(\rho^j)^\gamma
        \right)(\mathrm{x},t)d\mathrm{x}
        +\int^T_0\int_{\mathbb{R}^N}
        (\nu_1 h(\rho^j)|\nabla
        \mathrm{U}^j|^2)(\mathrm{x},t)d\mathrm{x}dt\nonumber\\
                &&+\alpha\varepsilon\int^T_0\int_{\mathbb{R}^N}
        ((\rho^j)^\theta|\nabla
        \mathrm{U}^j|^2)(\mathrm{x},t)d\mathrm{x}dt
        \leq        C,\ \mathrm{if} \ \gamma>1,\label{SSIC-E3.10}
            \end{eqnarray}
             \begin{eqnarray}
        &&\sup_{t\in[0,T]}\int_{\mathbb{R}^N}\left(
        \frac{1}{2}\rho^j|\mathrm{U}^j|^2+\rho^j\log\rho^j-\bar{\rho}\log\bar{\rho}-(\log\bar{\rho}+1)(\rho^j-\bar{\rho})
        \right)(\mathrm{x},t)d\mathrm{x}
        \nonumber\\
                &&+\int^T_0\int_{\mathbb{R}^N}
        (\nu_1 h(\rho^j)|\nabla
        \mathrm{U}^j|^2)(\mathrm{x},t)d\mathrm{x}dt+\alpha \varepsilon\int^T_0\int_{\mathbb{R}^N}
        ((\rho^j)^\theta|\nabla
        \mathrm{U}^j|^2)(\mathrm{x},t)d\mathrm{x}dt\nonumber\\
                &\leq&        C,\ \mathrm{if} \ \gamma=1,\
                \bar{\rho}=e^{-|\mathrm{x}|},\label{SSIC-E3.13-10}
            \end{eqnarray}
     \begin{eqnarray}
        &&\sup_{t\in[0,T]}\int_{B_{\varepsilon_j,R_j}}
        \frac{1}{2}\rho^j\left|\mathrm{U}^j+\frac{2h'(\rho^j)+\theta\varepsilon(\rho^j)^{
        \theta-1}}{\rho^j}\nabla\rho^j
        \right|^2(\mathrm{x},t)d\mathrm{x}\nonumber\\
                &&+\int^T_0\int_{B_{\varepsilon_j,R_j}}
        \left(\frac{2h'(\rho^j)+\theta\varepsilon(\rho^j)^{
        \theta-1}}{\rho^j}
        \nabla\rho^j\nabla(\rho^j)^{\gamma}       \right)(\mathrm{x},t)d\mathrm{x}dt
        \leq        C.\label{SSIC-E3.13-1}
            \end{eqnarray}
Moreover, the following uniform estimate hold
     \begin{equation}
        \sup_{t\in[0,T]}
        \|\sqrt{\rho^j}\|_{H^1(B_{\varepsilon_j,R_j})}\leq
        C.\label{SSIC-E3.12}
    \end{equation}
\end{lem}

From this lemma, we can obtain the following lemma.
\begin{lem}\label{SSIC-L3.2}
  The pressure $(\rho^j)^\gamma$ is bounded in $L^\frac{N+2}{N}(\mathbb{R}^N\times[0,T]
  )$ when $N\geq3$, in $L^\beta(\mathbb{R}^N\times[0,T]
  )$ for all $\beta\in[1,2)$ when $N=2$.
\end{lem}
\begin{proof}
From (\ref{SSIC-E2.5}), (\ref{SSIC-E3.9})-(\ref{SSIC-E3.13-1}), we
have $(\rho^j)^{\gamma/2}$ is bounded in
$L^2(0,T;H^1(B_{\varepsilon_j,R_j}))$.

When $N\geq3$, we get $(\rho^j)^{\gamma/2}$ is bounded in
$L^2(0,T;L^\frac{2N}{N-2}(B_{\varepsilon_j,R_j}))$ or
$(\rho^j)^{\gamma}$ is bounded in
$L^1(0,T;L^\frac{N}{N-2}(B_{\varepsilon_j,R_j}))$. Since
$(\rho^j)^{\gamma}$ is bounded in
$L^\infty(0,T;L^1(B_{\varepsilon_j,R_j}))$, H\"{o}lder's
inequality implies that $(\rho^j)^{\gamma}$ is bounded in
$L^\frac{N+2}{N}(B_{\varepsilon_j,R_j}\times[0,T]
  )$. From (\ref{SSIC-E3.7}), we obtain that $(\rho^j)^{\gamma}$ is bounded in
$L^\frac{N+2}{N}(\mathbb{R}^N\times[0,T]  )$.

  Similarly, we can get that $(\rho^j)^\gamma$ is bounded in  $L^\beta(\mathbb{R}^N\times[0,T]
  )$ for all $\beta\in[1,2)$ when $N=2$.
\end{proof}

 In the
following proposition, we will estimate
$\|\rho|\mathrm{U}|^{2}\ln(1+|\mathrm{U}|^{2})\|_{L^1(\mathbb{R}^N)}$.

\begin{prop}\label{SSIC-P3.1}
  If $\nu_1 h\leq2h+Ng\leq \nu_2 h$ and $\int^{\infty}_{0}\rho_0(1+|u_0|^{2})\ln(1+|u_0|^2)r^{N-1}dr\leq C$,
  then the following estimate is true
    \begin{equation}
      \sup_{t\in[0,T]}\int^{R_j}_{\varepsilon_j}\rho^j\frac{|u^j|^{2}}{2}\ln(1+|u^j|^2)r^{N-1}dr
      \leq C\label{SSIC-E3.13}
    \end{equation}
  where $C$ is a constant independent of $\varepsilon_j$, $R_j$
  and $\delta_j$.
\end{prop}
\begin{proof}
   Multiplying (\ref{SSIC-E3.4}) by
  $r^{N-1}u^j(1+\ln(1+|u^j|^2))$, integrating the resulting equation and using (\ref{SSIC-E3.3}) yield
    \begin{eqnarray*}
      &&\frac{d}{dt}\int^{R_j}_{\varepsilon_j}\rho^j\frac{1+|u^j|^2}{2}\ln(1+|u^j|^2)r^{N-1}dr
        \nonumber\\
                &&+\int^{R_j}_{\varepsilon_j}(2h+\varepsilon(\rho^j)^\theta)(1+\ln(1+|u^j|^2))(
                (u^j)_r^2+(N-1)\frac{(u^j)^2}{r^2})      r^{N-1}dr\nonumber\\
                &&
        +\int^{R_j}_{\varepsilon_j}(2h+\varepsilon(\rho^j)^\theta)\frac{2(u^j)^2}{1+|u^j|^2}
                (u^j)_r^2   r^{N-1}dr\nonumber\\
                &&+\int^{R_j}_{\varepsilon_j}(g+(\theta-1)\varepsilon(\rho^j)^\theta)(1+\ln(1+|u^j|^2))(
                (u^j)_r+(N-1)\frac{u^j}{r})^2      r^{N-1}dr\nonumber\\
                &&
        +\int^{R_j}_{\varepsilon_j}(g+(\theta-1)\varepsilon(\rho^j)^\theta)\frac{2(u^j)^2}{1+|u^j|^2}
                (u^j)_r((u^j)_r+(N-1)\frac{u_j}{r})
                r^{N-1}dr\nonumber\\
                    &&
                +\int^{R_j}_{\varepsilon_j}((\rho^j)^\gamma)_r(1+\ln(1+|u^j|^2))u^j
                r^{N-1}dr=0.
    \end{eqnarray*}
Since $\nu_1 h\leq2h+Ng\leq\nu_2 h$ and
$(1+N(\theta-1))\varepsilon(\rho^j)^\theta=\alpha\varepsilon(\rho^j)^\theta$,
we have
     \begin{eqnarray}
       &&\frac{d}{dt}\int^{R_j}_{\varepsilon_j}\rho^j\frac{1+|u^j|^2}{2}\ln(1+|u^j|^2)r^{N-1}dr
        \nonumber\\
        &&+\int^{R_j}_{\varepsilon_j}(\nu_1 h+\alpha\varepsilon(\rho^j)^\theta)(1+\ln(1+|u^j|^2))(
                (u^j)_r^2+(N-1)\frac{(u^j)^2}{r^2})      r^{N-1}dr\nonumber\\
                    &&
                +\int^{R_j}_{\varepsilon_j}((\rho^j)^\gamma)_r(1+\ln(1+|u^j|^2))u^j
                r^{N-1}dr\nonumber\\
                &\leq&C \int^{R_j}_{\varepsilon_j}( h+\varepsilon(\rho^j)^\theta)(
                (u^j)_r^2+(N-1)\frac{(u^j)^2}{r^2})
                r^{N-1}dr.\label{SSIC-E3.16-1}
    \end{eqnarray}
Using integration by parts and Young's inequality, we have
    \begin{eqnarray*}
      && \left|\int^{R_j}_{\varepsilon_j}((\rho^j)^\gamma)_r(1+\ln(1+|u^j|^2))u^j
                r^{N-1}dr
      \right|\\
                    &\leq&C\int^{R_j}_{\varepsilon_j}
                    |u^j_r|(1+\ln(1+|u^j|^2))(\rho^j)^\gamma r^{N-1}dr
                    +C\int^{R_j}_{\varepsilon_j}
                  |u^j|(1+\ln(1+|u^j|^2))(\rho^j)^\gamma r^{N-2}dr\\
                        &\leq&\frac{\nu_1}{2}\int^{R_j}_{\varepsilon_j}
                         h(1+\ln(1+|u^j|^2))(
                (u^j)_r^2+(N-1)\frac{(u^j)^2}{r^2})      r^{N-1}dr\nonumber\\
                    &&
                        +C\int^{R_j}_{\varepsilon_j}h^{-1}(\rho^j)^{2\gamma}(1+\ln(1+|u^j|^2))
                        r^{N-1}dr\\
    &\leq&\frac{\nu_1}{2}\int^{R_j}_{\varepsilon_j}
                         h(1+\ln(1+|u^j|^2))(
                (u^j)_r^2+(N-1)\frac{(u^j)^2}{r^2})      r^{N-1}dr\nonumber\\
                    &&
                        +C\left(\int^{R_j}_{\varepsilon_j}\left(\frac{(\rho^j)^{
                        2\gamma-\frac{\delta}{2}}}{h}\right)^{\frac{2}{2-\delta}}
                        r^{N-1}dr\right)^{\frac{2-\delta}{2}}
                        \left(\int^{R_j}_{\varepsilon_j}\rho^j(1+\ln(1+|u^j|^2))^{\frac{2}{\delta}}
                        r^{N-1}dr\right)^{\frac{\delta}{2}}.
    \end{eqnarray*}
Combining it with (\ref{SSIC-E3.9})-(\ref{SSIC-E3.13-10}) and
(\ref{SSIC-E3.16-1}), we get
    \begin{eqnarray}
      &&\sup_{t\in[0,T]}\int^{R_j}_{\varepsilon_j}\rho^j\frac{1+|u^j|^2}{2}\ln(1+|u^j|^2)r^{N-1}dr
      \nonumber\\
      &\leq& C+ C_\delta\left(\int^{R_j}_{\varepsilon_j}\left(\frac{(\rho^j)^{
                        2\gamma-\frac{\delta}{2}}}{h}\right)^{\frac{2}{2-\delta}}
                        r^{N-1}dr\right)^{\frac{2-\delta}{2}}.\label{SSIC-E3.17-1}
    \end{eqnarray}
From (\ref{SSIC-E2.5}), we have $h\geq\nu \rho$ and
    $$
\left(\int^{R_j}_{\varepsilon_j}\left(\frac{(\rho^j)^{
                        2\gamma-\frac{\delta}{2}}}{h}\right)^{\frac{2}{2-\delta}}
                        r^{N-1}dr\right)^{\frac{2-\delta}{2}}
                        \leq C\left(\int^{R_j}_{\varepsilon_j}\left((\rho^j)^{
                        2\gamma-1-\frac{\delta}{2}}\right)^{\frac{2}{2-\delta}}
                        r^{N-1}dr\right)^{\frac{2-\delta}{2}}.
    $$
Then, using Lemma \ref{SSIC-L3.2}, we check that the right hand
side is bounded $L^1$ in time for some small $\delta$, without any
condition when $N=2$, and when $N\geq3$ under the condition that
    $$2\gamma-1<\frac{N+2}{N}\gamma,$$
which gives rise to the restriction $\gamma<\frac{N}{N-2}$. In
either cases, we have
    $$
\sup_{t\in[0,T]}\int^{R_j}_{\varepsilon_j}\rho^j\frac{1+|u^j|^2}{2}\ln(1+|u^j|^2)r^{N-1}dr
\leq C.
    $$

When $N\geq3$ and $\gamma\geq \frac{N}{N-2}$, we need the extra
hypothesis (\ref{SSIC-E2.8}) to show that the right hand side of
(\ref{SSIC-E3.17-1}) is bounded and to obtain the same result.
\end{proof}

From (\ref{SSIC-E3.7})-(\ref{SSIC-E3.8}), we deduce that
\begin{cor}\label{SSIC-C3.1}
    \begin{equation}
      \sup_{t\in[0,T]}\int_{\mathbb{R}^N}\rho^j\frac{|\mathrm{U}^j|^{2}}{2}\ln(1+|\mathrm{U}^j|^2)d\mathrm{x}
      \leq C.
    \end{equation}
\end{cor}

\begin{prop}\label{SSIC-P3.2}
The sequence $\{\rho^j\}$ is bounded in
$L^\infty(0,T;L^\frac{N}{N-2}(\mathbb{R}^N))$ when $N\geq3$, or
$L^\infty(0,T;L^q(\mathbb{R}^2))$ for all $q\geq1$. There exists a
subsequence of $\{\rho^j\}$, still denoted by itself, such that
    \begin{equation}
      \rho^j(\mathrm{x},t)\rightarrow
            \rho(\mathrm{x},t),\label{SSIC-E3.16}
    \end{equation}
strongly in $C([0,T];L^\beta_{\mathrm{loc}}(\mathbb{R}^N))$,
$\beta\in[1,\frac{N}{N-2})$, as $j\rightarrow\infty$. Here,
$\rho\in L^\infty(0,T;L^1\cap L^\frac{N}{N-2}(\mathbb{R}^N))$ when
$N\geq3$, or $\rho\in L^\infty(0,T;L^q(\mathbb{R}^2))$ for all
$q\geq1$. Moreover, $\rho(\mathrm{x},t)=\rho(r,t)$ is a
spherically symmetric function.
\end{prop}
\begin{proof} We only consider the case $N\geq3$, since the
proof of the case that $N=2$ is similar.

It follows from (\ref{SSIC-E3.7}) and (\ref{SSIC-E3.12}) that
$\{\sqrt{\rho^j}\}$ is bounded in
$L^\infty(0,T;L^q(\mathbb{R}^N))$ for $q\in[2,\frac{2N}{N-2}]$.
Thus, $\{\rho^j\}$ is bounded in
$L^\infty(0,T;L^\frac{N}{N-2}(\mathbb{R}^N))$ and
$\{\rho^j\mathrm{U}^j\}$ is bounded in
$L^\infty(0,T;L^\frac{N}{N-1}(\mathbb{R}^N))$ due to
(\ref{SSIC-E3.10})-(\ref{SSIC-E3.13-10}). Then, the continuity
equation yields that
$\{\partial_t\rho^j\}_{\varepsilon_j\leq\frac{1}{k},R_j\geq {n}}$
is bounded in
$L^\infty(0,T;W^{-1,\frac{N}{N-1}}(B_{\frac{1}{k},{n}}))$, for any
$k\geq n^{2N}$. Moreover, since $\nabla
\rho^j=2\sqrt{\rho^j}\nabla \sqrt{\rho^j}$, we have
$\{\nabla\rho^j\}_{\varepsilon_j\leq\frac{1}{k},R_j\geq {n}}$ is
bounded in $L^\infty(0,T;L^{\frac{N}{N-1}}(B_{\frac{1}{k},{n}}))$.
From the Aubin-Lions lemma, we get
    \begin{equation}
      \rho^j(\mathrm{x},t)\rightarrow
            \rho(\mathrm{x},t),
            \ \textrm{strongly in }C([0,T];L^\frac{N}{N-1}(B_{\frac{1}{k},{n}})),
            \ \textrm{as }j\rightarrow\infty.
    \end{equation}
 Since
    $$
    \|\rho^j-\rho\|_{L^\infty([0,T];L^\frac{N}{N-1}(B_{{n}}))}
    \leq   \frac{C}{k}  \|\rho^j-\rho\|_{L^\infty([0,T];L^\frac{N}{N-2}(B_{\frac{1}{k}}))}
    +
    \|\rho^j-\rho\|_{L^\infty([0,T];L^\frac{N}{N-1}(B_{\frac{1}{k},{n}}))},
    $$
 we get
    \begin{equation}
      \rho^j(\mathrm{x},t)\rightarrow
            \rho(\mathrm{x},t),
            \ \textrm{strongly in }C([0,T];L^\frac{N}{N-1}(B_{{n}})),
            \ \textrm{as }j\rightarrow\infty.\label{SSIC-E3.20}
    \end{equation}
 Clearly, (\ref{SSIC-E3.16}) holds and $\rho(\mathrm{x},t)$ is
spherically symmetric.
\end{proof}

From (\ref{SSIC-E3.7}), Lemma \ref{SSIC-L3.2} and Proposition
\ref{SSIC-P3.2}, we immediately obtain the following lemma.

\begin{lem}\label{SSIC-L3.3}
 There exists a
subsequence of $\{\rho^j\}$, still denoted by itself, such that
    \begin{equation}
      (\rho^j)^\gamma(\mathrm{x},t)\rightarrow
            \rho^\gamma(\mathrm{x},t),
    \end{equation}
strongly in $L^1_{loc}(\mathbb{R}^N\times[0,T])$, as
$j\rightarrow\infty$.
\end{lem}

\begin{prop}\label{SSIC-P3.3}
For any $k\geq n^{2N}$, there exists a subsequence of
$\{\rho^j\}_{\varepsilon_j\leq\frac{1}{k},R_j\geq {n}}$, still
denoted by itself, such that
    \begin{equation*}
      \nabla\sqrt{\rho^j(\mathrm{x},t)}\stackrel{\ast}{\rightharpoonup}
            \nabla\sqrt{\rho(\mathrm{x},t)},
            \  \textrm{ weak-}\ast \textrm{ in }
                  L^\infty([0,T],L^{2}(B_{\frac{1}{k},{n}})),
    \end{equation*}
     \begin{equation*}
      \nabla\bar{h}(\rho^j(\mathrm{x},t))\stackrel{\ast}{\rightharpoonup}
            \nabla\bar{h}(\rho(\mathrm{x},t)),
            \  \textrm{ weak-}\ast \textrm{ in }
                  L^\infty([0,T],L^{2}(B_{\frac{1}{k},{n}})),
    \end{equation*}
as $j\rightarrow\infty$, where $\bar{h}$ satisfies $\bar{h}(0)=0$
and $\bar{h}'(s)=\frac{h'(s)}{\sqrt{s}}$. Moreover,
$\nabla\sqrt{\rho}\in L^\infty([0,T],L^{2}(\mathbb{R}^N))$ and
$\nabla \bar{h}(\rho)\in L^\infty([0,T],L^{2}(\mathbb{R}^N))$.
\end{prop}
\begin{proof}
  It follows from  (\ref{SSIC-E3.12}) that
$\{\nabla\sqrt{\rho^j}\}_{\varepsilon_j\leq\frac{1}{k},R_j\geq
{n}}$ is bounded in $L^\infty(0,T;L^2(B_{\frac{1}{k},{n}}))$.
Thus, there exists a function $f\in
L^\infty(0,T;L^2(B_{\frac{1}{k},{n}}))$ such that, up to a
subsequence,
    \begin{equation}
      \nabla\sqrt{\rho^j(\mathrm{x},t)}\stackrel{\ast}{\rightharpoonup}
            f,
            \  \textrm{ weak-}\ast \textrm{ in }
                  L^\infty([0,T],L^{2}(B_{\frac{1}{k},{n}})).
    \end{equation}
Combining it with (\ref{SSIC-E3.16}), one can easily obtain
$f=\nabla\sqrt{\rho}$ and
    $$
    \|\nabla\sqrt{\rho}\|_{L^\infty([0,T],L^{2}(B_{\frac{1}{k},{n}}))}
    \leq\liminf_{j\rightarrow\infty}
        \|\nabla\sqrt{\rho^j}\|_{L^\infty([0,T],L^{2}(B_{\frac{1}{k},{n}}))}
        \leq C
    $$
with a constant $C$ independent of $k$ and $n$. Clearly, we have
$\nabla\sqrt{\rho}\in L^\infty([0,T],L^{2}(\mathbb{R}^N))$.
Similarly, we can easily obtain the result for $\bar{h}$.
\end{proof}

From Propositions \ref{SSIC-P3.1}-\ref{SSIC-P3.2} and Corollary
\ref{SSIC-C3.1}, using  similar arguments as that in the proof of
Lemmas 4.4 and 4.6 in \cite{Mellet2007}, we can obtain the
following proposition.

\begin{prop}\label{SSIC-P3.4}
  1) Up to a subsequence, $\mathrm{m}^j=\rho^j\mathrm{U}^j$
  converges strongly in $L^1_{loc}(\mathbb{R}^N\times[0,T])$
  and $L^2(0,T;L^{\beta}_{loc}(\mathbb{R}^N))$ to some
  $\mathrm{m}(\mathrm{x},t)$,  for all $\beta\in[1,\frac{N}{N-1})$.

  2) $\sqrt{\rho^j}\mathrm{U}^j$ converges strongly in
  $L^2_{loc}(\mathbb{R}^N\times[0,T])$ to
  $\frac{\mathrm{m}}{\sqrt{\rho}}$ (defined to be zero when
  $\mathrm{m}=0$). In particular,
  $\mathrm{m}(\mathrm{x},t)=0$ a.e. on $\{
  \rho(\mathrm{x},t)=0\}$ and there exists a function
  $\mathrm{U}(\mathrm{x},t)$ such that
    $$
    \mathrm{m}(\mathrm{x},t)=\rho(\mathrm{x},t)\mathrm{U}(\mathrm{x},t).
    $$
\end{prop}
\begin{proof}
  We only consider the case $N\geq3$, since the
proof of the case that $N=2$ is similar.

1) Since $\{\sqrt{\rho^j}\}$ is bounded in $L^\infty(0,T;L^2\cap
L^\frac{2N}{N-2})$ and $\{\sqrt{\rho^j}\mathrm{U}^j\}$ is bounded
in $L^\infty(0,T;L^2)$, we have that
    \begin{equation}
    \{{\rho^j}\mathrm{U}^j\}\textrm{ is bounded in }
    L^\infty(0,T;L^1\cap L^{\frac{N}{N-1}}(\mathbb{R}^N)).\label{SSIC-E3.25-1}
    \end{equation}

Since $\nabla(\rho^j
\mathrm{U}^j)=2\sqrt{\rho^j}\mathrm{U}^j\nabla\sqrt{\rho^j}+\sqrt{\rho^j}\sqrt{\rho^j}\nabla\mathrm{U}^j$,
from (\ref{SSIC-E2.5}) and (\ref{SSIC-E3.9})-(\ref{SSIC-E3.13-1}),
we obtain that $\{\nabla(\rho^j
\mathrm{U}^j)\}_{\varepsilon_j\leq\frac{1}{k},R_j\geq {n}}$ is
bounded in $L^2(0,T;L^{1}(B_{\frac{1}{k},{n}}))$. In particular,
we get
    $$
    \{(\rho^j
\mathrm{U}^j)\}_{\varepsilon_j\leq\frac{1}{k},R_j\geq {n}}\
\textrm{ is bounded in } L^2(0,T;W^{1,1}(B_{\frac{1}{k},{n}})).
    $$

Since $\{\rho^j \}_{\varepsilon_j\leq\frac{1}{k},R_j\geq {n}}$ is
bounded in $L^\infty(B_{\frac{1}{k},{n}}\times[0,T])$, from
(\ref{SSIC-E3.2}), we can obtain that
    $$
    \{\partial_t(\rho^j
\mathrm{U}^j)\}_{\varepsilon_j\leq\frac{1}{k},R_j\geq {n}}\
\textrm{ is bounded in }
L^2(0,T;W^{-2,\frac{N}{N-1}}(B_{\frac{1}{k},{n}})).
    $$
From the Aubin-Lions lemma, we have
     \begin{equation*}
      \rho^j\mathrm{U}^j\rightarrow
            \mathrm{m},
    \end{equation*}
strongly in $L^2([0,T];L^\beta(B_{\frac{1}{k},{n}}))$ for all
$\beta\in[1,\frac{N}{N-1})$. From (\ref{SSIC-E3.25-1}), we can
easily obtain that
    \begin{equation*}
      \rho^j\mathrm{U}^j\rightarrow
            \mathrm{m},
    \end{equation*}
strongly in $L^2([0,T];L^\beta(B_{{n}}))$ for all
$\beta\in[1,\frac{N}{N-1})$.

2) Using the similar argument as that in the proof of Lemma 4.6 in
\cite{Mellet2007}, we can obtain the part 2) of Proposition
\ref{SSIC-P3.4}, where
    $$
    \mathrm{U}=\left\{
    \begin{array}{ll}
    \frac{\mathrm{m}}{\rho},&\mathrm{if }\rho\not=0,\\
    0,&\mathrm{if }\rho=0,
    \end{array}
    \right.
    $$
 and omit the detail.
\end{proof}

Then, using  similar arguments as that in the proof of Corollary
4.2 in \cite{Guo2007}, we can obtain the following corollary and
omit the details.

\begin{cor}\label{SSIC-C3.2}
  Let $m^j(r,t)=(\rho^ju^j)(r,t)$, then

  1) there exists a function $m(r,t)$ such that
  $\mathrm{m}(\mathrm{x},t)=m(r,t)\frac{\mathrm{x}}{r}$ and $m^j(r,t)$
  converges to $m(r,t)$ strongly in
  $L^2(0,T;L^{\beta}_{\mathrm{loc}}([0,\infty);r^{N-1}dr))$ for all $\beta\in[1,\frac{N}{N-1})$;

  2) there exits a function $u(r,t)$ such that
  $\mathrm{U}(\mathrm{x},t)=u(r,t)\frac{\mathrm{x}}{r}$ and
  $\sqrt{\rho^j}u^j$ converges to $\frac{m}{\sqrt{\rho}}$ (defined to be  zero when $m=0$)
  strongly in $L^2(0,T;L^2_{\mathrm{loc}}([0,\infty);r^{N-1}dr))$.
\end{cor}

Now, we show that $(\rho,\mathrm{U})$ obtained in Propositions
\ref{SSIC-P3.1}-\ref{SSIC-P3.4} satisfies the weak form of
(\ref{SSIC-E2.1}), that is (\ref{SSIC-E2.10}) holds.

\begin{prop}
Let $(\rho,\mathrm{U})$ be the limit described as  in Propositions
\ref{SSIC-P3.1}-\ref{SSIC-P3.4}. Then  (\ref{SSIC-E2.10}) holds.
Moreover,  $\rho\in C([0,\infty);L^{1}(\mathbb{R}^N))$.
\end{prop}
\begin{proof}
 We only consider the case $t_1>0$, since the
proof of the case that $t_1=0$ is similar.

At first, we derive the weak form of (\ref{SSIC-E3.3}). For any
$\varphi\in C^1_{c}([0,\infty)\times[0,\infty))$, there exists
$n>0$ such that supp$\varphi(\cdot,t)\subset[0,n]$. It follows
from (\ref{SSIC-E3.3}), (\ref{SSIC-E3.6-10}) and
(\ref{SSIC-E3.7})-(\ref{SSIC-E3.8}) that
    \begin{equation}
      \int^\infty_0\rho^j\varphi
      r^{N-1}dr|^{t_2}_{t_1}-\int^{t_2}_{t_1}\int^\infty_0(\rho^j\varphi_t+\rho^ju^j\varphi_r)r^{N-1}drdt=0,
      \label{SSIC-E3.21}
    \end{equation}
for any $j$ satisfying $R_j\geq n$. From Proposition
\ref{SSIC-P3.2}, we have
    $$
     \int^\infty_0\rho^j\varphi
      r^{N-1}dr\rightarrow  \int^\infty_0\rho\varphi
      r^{N-1}dr,
    $$
and
        $$
    \int^{t_2}_{t_1}\int^\infty_0\rho^j\varphi_tr^{N-1}drdt
    \rightarrow
        \int^{t_2}_{t_1}\int^\infty_0\rho\varphi_tr^{N-1}drdt
        $$
        as $j\rightarrow0$.
From Proposition \ref{SSIC-P3.2} and Corollary \ref{SSIC-C3.2}, we
have
    \begin{eqnarray}
      &&\int^{t_2}_{t_1}\int^\infty_0\rho^ju^j\varphi_rr^{N-1}drdt=
            \int^{t_2}_{t_1}\int^\infty_0\sqrt{\rho^j}(\sqrt{\rho^j}
            u^j)\varphi_rr^{N-1}drdt\rightarrow\nonumber\\
      &&            \int^{t_2}_{t_1}\int^\infty_0\sqrt{\rho}(\sqrt{\rho}
            u)\varphi_rr^{N-1}drdt=\int^{t_2}_{t_1}\int^\infty_0\rho u\varphi_rr^{N-1}drdt,
            \label{SSIC-E3.22-1}
    \end{eqnarray}
as $j\rightarrow0$.

Therefore, taking  limit $j\rightarrow\infty$ in
(\ref{SSIC-E3.21}), we get
    \begin{equation}
      \int^\infty_0\rho\varphi
      r^{N-1}dr|^{t_2}_{t_1}-\int^{t_2}_{t_1}\int^\infty_0(\rho\varphi_t+\rho
      u\varphi_r)r^{N-1}drdt=0.
      \label{SSIC-E3.22}
    \end{equation}
For any $\phi_1\in C^1_c(\mathbb{R}^N\times[t_1,t_2])$, define
    $$\varphi(r,t)=\int_S\phi_1(ry,t)dS_y,$$
where the integral is over the unit sphere $S=S^{N-1}$ in
$\mathbb{R}^N$. Then is follows from (\ref{SSIC-E3.22}) that
(\ref{SSIC-E2.10}) holds.

Similarly, we can easily obtain
    $$
    \partial_t\sqrt{\rho}+\mathrm{div}(\sqrt{\rho}u)-\frac{1}{2}\mathcal{Q}=0,
    \ \textrm{ in } \ \mathcal{D'},
    $$
where $\mathcal{Q}\in L^2(0,T;L^2(\mathbb{R}^N))$ is the weak
limit of $\{\sqrt{\rho^j}\mathrm{div}u^j\}$ in
$L^2(0,T;L^2(\mathbb{R}^N))$. Thus, we have $\sqrt{\rho}_t\in
L^2([0,\infty);H^{-1}(\mathbb{R}^N))$. Since $\sqrt{\rho}\in
L^\infty([0,\infty);H^{1}(\mathbb{R}^N))$,  we can easily get that
$\sqrt{\rho}\in C([0,\infty);L^{{2}}(\mathbb{R}^N))$ .
\end{proof}
In the following, we prove that $(\rho,\mathrm{U})$ satisfies
(\ref{SSIC-E2.11})

\begin{prop}
  Let $(\rho,\mathrm{U})$ be the limit described as  in Propositions
\ref{SSIC-P3.1}-\ref{SSIC-P3.4}. Then  (\ref{SSIC-E2.11}) holds.
\end{prop}
\begin{proof}
   For any
$\phi\in C^2_{c}([0,\infty)\times[0,T])$ with
$\phi(0,t)=\phi(r,T)=0$, there exists $n>0$ such that
$\mathrm{supp}\phi(\cdot,t)\subset [0, {n}]$. It follows from
(\ref{SSIC-E3.4})-(\ref{SSIC-E3.6-10}) and
(\ref{SSIC-E3.7})-(\ref{SSIC-E3.8}) that
    \begin{eqnarray}
        &&  \int^\infty_{\varepsilon_j}\rho^j_0u^j_0\phi(r,0)
      r^{N-1}dr\nonumber\\
        &&+\int^{T}_{0}\int^\infty_0(\rho^ju^j\phi_t+\rho^j(u^j)^2\phi_r+(\rho^j)^\gamma(\phi_r+\frac{(N-1)\phi}{r}))r^{N-1}drdt
      \nonumber\\
            &&-\int^T_0\int^\infty_{\varepsilon_j}2h(\rho^j)(u^j_r\phi_r+\frac{(N-1)u^j\phi}{r^2})r^{N-1}drdt\nonumber\\
      &=&-\int^T_0\int^\infty_{\varepsilon_j}\varepsilon_j(\rho^j)^\theta
            (\frac{(N-1)u^j_r\phi}{r}+\frac{(N-1)u^j\phi_r}{r}+\frac{(N-1)(N-2)u^j\phi}{r^2})r^{N-1}drdt
            +\varepsilon_b^j
      \nonumber\\
            &&
    +\int^T_0\int^\infty_{\varepsilon_j}(g(\rho^j)+\theta\varepsilon_j(\rho^j)^\theta)(u^j_r+\frac{(N-1)u^j}{r^2})
    (\phi_r+\frac{(N-1)\phi}{r})r^{N-1}drdt,
      \label{SSIC-E3.23}
    \end{eqnarray}
for any $j$ satisfying $\varepsilon_j\leq \frac{1}{k}$ and
$R_j\geq {n}$, where
    \begin{equation}
      \varepsilon_b^j=\int^T_0\left\{
      [(2h(\rho^j)+g(\rho^j)+\theta\varepsilon_j(\rho^j)^\theta) u^j_r]
      (\varepsilon_j,t)\varepsilon^{N-1}_j\phi(\varepsilon_j,t)-
        \varepsilon_j^{N-1}(\rho^j)^\gamma(\varepsilon_j,t)\phi(\varepsilon_j,t)
      \right\}dt.
    \end{equation}

\noindent\textbf{Claim:}
    \begin{equation}
      \lim_{\varepsilon_j\rightarrow0+}\varepsilon_b^j=0.\label{SSIC-E3.25}
    \end{equation}
Since
    \begin{eqnarray*}
      &&\left|\varepsilon_j^{N-1}\int^T_0[(\rho^j)^\gamma\phi](\varepsilon_j,t)dt
      \right|\leq\max_{t\in[0,T]}|\phi(\varepsilon_j,t)|\varepsilon_j^{N-1}
      \int^T_0(\rho^j)^\gamma(\varepsilon_j,t)dt\\
            &\leq&C\max_{t\in[0,T]}|\phi(\varepsilon_j,t)|
      \int^T_0 \int^{R_j}_{\varepsilon_j}[(\rho^j)^\gamma
     + |\partial_r(\rho^j)^\gamma|]r^{N-1}drdt,
    \end{eqnarray*}
        $$
         \int^T_0 \int^{R_j}_{\varepsilon_j}(\rho^j)^\gamma
     r^{N-1}drdt\leq C,
        $$
        $$
 \int^T_0 \int^{R_j}_{\varepsilon_j}
      |\partial_r(\rho^j)^\gamma|r^{N-1}drdt\leq  C\int^T_0 \int^{R_j}_{\varepsilon_j}[(\rho^j)^\gamma
     + |\partial_r(\rho^j)^\frac{\gamma}{2}|^2]r^{N-1}drdt\leq C,
        $$
     and
     $\lim_{\varepsilon_j\rightarrow0+}\max_{t\in[0,T]}|\phi(\varepsilon_j,t)|=0$,
     we have
        \begin{equation}
          \lim_{\varepsilon_j\rightarrow0+}\varepsilon_j^{N-1}\int^T_0[(\rho^j)^\gamma\phi]
          (\varepsilon_j,t)dt=0.
        \end{equation}
From (\ref{SSIC-E3.3}) and $u(\varepsilon_j,t)=0$, we get
    $$
        \rho^j_t(\varepsilon_j,t)+\rho^j(\varepsilon_j,t)\partial_ru^j(\varepsilon_j,t)=0.
    $$
Thus, using (\ref{SSIC-E2.4}), we have
    \begin{eqnarray*}
      &&\varepsilon_j^{N-1}\int^T_0((2h(\rho^j)+g(\rho^j))u^j_r\phi)(\varepsilon_j,t)dt=
      -\varepsilon_j^{N-1}\int^T_0(\frac{2h(\rho^j)+g(\rho^j)}{\rho}\partial_t\rho^j\phi)(\varepsilon_j,t)dt\\
      &=&
      -\varepsilon_j^{N-1}\int^T_0(2\partial_th(\rho^j)\phi)(\varepsilon_j,t)dt\\
            &=&2\varepsilon_j^{N-1}h(\rho^j_0(\varepsilon_j))\phi(\varepsilon_j,0)
            +2\varepsilon_j^{N-1}\int^T_0(h(\rho^j)\partial_t\phi)(\varepsilon_j,t)dt.
    \end{eqnarray*}
It is easy to obtain
    $$
    |\sqrt{\rho^j(\varepsilon_j,t)}|\leq
    C\|\sqrt{\rho^j}\|_{H^1([\varepsilon_j,R_j])}\leq
    C\varepsilon_j^{-\frac{N-1}{2}},
    $$
    $$
    |\sqrt{\rho^j(1,t)}|\leq
    C\|\sqrt{\rho^j}\|_{H^1([1,R_j])}\leq
    C,
    $$
        \begin{equation}
        |\bar{h}(\rho^j(\varepsilon_j,t))|\leq
        C|\bar{h}(\rho^j(1,t))|+\|\nabla \bar{h}\|_{L^2([\varepsilon^j,1])}
        \leq C+C\varepsilon_j^{-\frac{N-1}{2}},\label{SSIC-E3.33-1}
        \end{equation}
where $\bar{h}$ satisfies $\bar{h}(0)=0$ and
$\bar{h}'(s)=\frac{h'(s)}{\sqrt{s}}$. Since $h(s)\leq
\sqrt{s}\bar{h}(s)$, we have
    $$
    h(\rho^j(\varepsilon_j,t))\leq C+C\varepsilon_j^{-(N-1)}.
    $$
Thus, we can easily obtain
    $$
    \lim_{\varepsilon_j\rightarrow0+}\max_{t\in[0,T]}|\partial_t\phi(\varepsilon_j,t)|=0
    $$
and
    \begin{eqnarray*}
    \varepsilon_j^{N-1}\left|\int^T_0(h(\rho^j)\partial_t\phi)(\varepsilon_j,t)dt
    \right|\leq
    C\max_{t\in[0,T]}|\partial_t\phi(\varepsilon_j,t)|\rightarrow0,
    \end{eqnarray*}
as $\varepsilon_j\rightarrow0+$. Hence, we have
    $$
    \lim_{\varepsilon_j\rightarrow0+}\varepsilon_j^{N-1}\int^T_0((2h(\rho^j)+g(\rho^j))u^j_r\phi)(\varepsilon_j,t)dt=0.
    $$
Similarly, one can obtain that
    $$
 \lim_{\varepsilon_j\rightarrow0+}\varepsilon_j^N\int^T_0
 [(\rho^j)^\theta u^j_r\phi](\varepsilon_j,t)dt=0.
    $$
 Thus, (\ref{SSIC-E3.25}) holds.

Now, for any $\phi_2\in (C^2_c(\mathbb{R}^N\times[0,T]))^N$ with
supp$\phi_2(\cdot,t)\subset B_{{n}}$ and $\phi_2(\mathrm{x},T)=0$,
we set
    \begin{equation}
      \phi(r,t)=\int_S\phi_2(ry,t)\cdot ydS_y.
    \end{equation}
Since
    $$
    (r^{N-1}\phi)_r=\partial_r\int_{|\mathrm{x}|\leq
    r}\mathrm{div}\phi_2(\mathrm{x},t)d\mathrm{x}=r^{N-1}\int_S(\phi_2^i)_{x_i}(ry,t)dS_y,
    $$
we have by direct calculation that
    $$
    -\int^T_0\int^\infty_{\varepsilon_j}2h(\rho^j)(u^j_r\phi_r+\frac{(N-1)u^j\phi}{r^2})r^{N-1}drdt
    =-\int^T_0\int_{|\mathrm{x}|> {\varepsilon_j}}
    2h(\rho^j)D(\mathrm{U}^j):\nabla\phi_2d\mathrm{x}dt.
    $$
Similarly, one has
    \begin{eqnarray*}
        &&
    \int^t_0\int^\infty_{\varepsilon_j}(g(\rho^j)+\theta\varepsilon_j(\rho^j)^\theta)(u^j_r+\frac{(N-1)u^j}{r})
    (\phi_r+\frac{N-1}{r}\phi)r^{N-1}drdt\\
        &=&
       \int^t_0\int_{|\mathrm{x}|> {\varepsilon_j}}(g(\rho^j)+\theta\varepsilon_j(\rho^j)^\theta)\mathrm{div
       U}^j\mathrm{div}
    \phi_2d\mathrm{x}dt
     \end{eqnarray*}
and
    \begin{eqnarray*}
        &&\int^T_0\int^\infty_{\varepsilon_j}(\rho^j)^\theta
            (\frac{(N-1)u^j_r\phi}{r}+\frac{(N-1)u^j\phi_r}{r}+\frac{(N-1)(N-2)u^j\phi}{r^2})r^{N-1}drdt\\
                & =&\int^t_0\int_{|\mathrm{x}|> {\varepsilon_j}}(\rho^j)^\theta(\mathrm{div
       U}^j\mathrm{div}
    \phi_2-D(\mathrm{U}^j):\nabla\phi_2)d\mathrm{x}dt.
    \end{eqnarray*}

Thus, from (\ref{SSIC-E3.23}), we have
    \begin{eqnarray}
        &&  \int_{|\mathrm{x}|> {\varepsilon_j}}\rho^j_0\mathrm{U}^j_0\cdot\phi_2(\mathrm{x},0)
      d\mathrm{x}+\int^{T}_{0}\int_{\mathbb{R}^N}
      (\sqrt{\rho^j}\sqrt{\rho^j}\mathrm{U}^j\cdot\partial_t\phi_2
      +\sqrt{\rho^j}\mathrm{U}^j\otimes\sqrt{\rho^j}\mathrm{U}^j:\nabla\phi_2
    )d\mathrm{x}dt
      \nonumber\\
      &&+\int^{T}_{0}\int_{\mathbb{R}^N}(\rho^j)^\gamma\mathrm{div}\phi_2d\mathrm{x}dt
            -\int^T_0\int_{|\mathrm{x}|> {\varepsilon_j}}
            [2h(\rho^j)D(\mathrm{U}^j):\nabla\phi_2+g(\rho^j)\mathrm{div
       U}^j\mathrm{div}
    \phi_2]d\mathrm{x}dt\nonumber\\
                           &=&
     \varepsilon_j  \int^T_0\int_{|\mathrm{x}|> {\varepsilon_j}}
     [(\theta-1)(\rho^j)^\theta\mathrm{div
       U}^j\mathrm{div}
    \phi_2+(\rho^j)^\theta
            D(\mathrm{U}^j):\nabla\phi_2]d\mathrm{x}dt+\varepsilon_b^j    .
      \label{SSIC-E3.28}
    \end{eqnarray}

We proceed to show that each term on the left hand side of
(\ref{SSIC-E3.28}) converges to corresponding term in
(\ref{SSIC-E2.11}), and each term on the right hand side of
(\ref{SSIC-E3.28}) vanishes as $j\rightarrow\infty$.

First, the proof of the convergence of
$\rho^j\mathrm{U}^j\partial_t\phi_2$ is similar to that of
(\ref{SSIC-E3.22-1}).

    Next,  from
    Proposition \ref{SSIC-P3.4}, we obtain
        \begin{equation*}
          \int^T_0\int_{\mathbb{R}^N}\sqrt{\rho^j}\mathrm{U}^j\otimes\sqrt{\rho^j}\mathrm{U}^j:\nabla\phi_2d\mathrm{x}dt
          \rightarrow
           \int^T_0\int_{\mathbb{R}^N}\sqrt{\rho}\mathrm{U}\otimes\sqrt{\rho}\mathrm{U}:\nabla\phi_2d\mathrm{x}dt,
           \ \mathrm{as}\ j\rightarrow\infty.
        \end{equation*}

From Lemma \ref{SSIC-L3.3}, we have
    \begin{equation*}
      \int^{T}_{0}\int_{\mathbb{R}^N}(\rho^j)^\gamma\mathrm{div}\phi_2d\mathrm{x}dt
      \rightarrow
      \int^{T}_{0}\int_{\mathbb{R}^N}\rho^\gamma\mathrm{div}\phi_2d\mathrm{x}dt,
      \ \textrm{as }j\rightarrow\infty.
    \end{equation*}

Concerning the diffusion terms on the left hand side of
(\ref{SSIC-E3.28}), using (\ref{SSIC-E3.7}) and integration by
parts, we have
    \begin{eqnarray}
     &&\int^T_0\int_{|\mathrm{x}|>\varepsilon_j}2h(\rho^j)D(\mathrm{U}^j):\nabla\phi_2d\mathrm{x}dt\nonumber\\
      &=&-\int^T_0\int_{\mathbb{R}^N}\left[\frac{h(\rho^j)}{\sqrt{\rho^j}}(\sqrt{\rho^j}
      \mathrm{U}^j)\cdot\Delta\phi_2
      +\frac{h(\rho^j)}{\sqrt{\rho^j}}(\sqrt{\rho^j}
      \mathrm{U}^j)\cdot\nabla\mathrm{div}\phi_2\right]d\mathrm{x}dt\nonumber\\
            &&
            -\int^T_0\int_{B_{\varepsilon_j,{n}}}[(\sqrt{\rho^j}\mathrm{U}^j)\cdot
            (\nabla\bar{h}(\rho^j)\cdot\nabla)\phi_2
            +(\sqrt{\rho^j}\mathrm{U}^j)\cdot
            (\nabla\phi_2\cdot\nabla)\bar{h}(\rho^j)]d\mathrm{x}dt.\label{SSIC-E3.33}
    \end{eqnarray}
Using the similar argument as that in the proof of
(\ref{SSIC-E3.33-1}), we have
    $$
    \left\|\frac{h(\rho^j)}{\sqrt{\rho^j}}
    \right\|_{L^\infty([0,T];L^2(B_{{n}}))}\leq C_n,
    \ \textrm{ and } \|\rho^j\|_{L^\infty([\frac{1}{k},{n}]\times[0,T])}\leq C_{k,n}.
    $$
Then, using the similar argument  as that in the proof of
(\ref{SSIC-E3.22-1}),  we have
    \begin{equation}
    \int^T_0\int_{\mathbb{R}^N}\frac{h(\rho^j)}{\sqrt{\rho^j}}(\sqrt{\rho^j}\mathrm{U}^j)\cdot\Delta\phi_2d\mathrm{x}dt
    \rightarrow
    \int^T_0\int_{\mathbb{R}^N}\frac{h(\rho)}{\sqrt{\rho}}(\sqrt{\rho}\mathrm{U})\cdot\Delta\phi_2d\mathrm{x}dt,
    \end{equation}
and
     \begin{equation}
    \int^T_0\int_{\mathbb{R}^N}\frac{h(\rho^j)}{\sqrt{\rho^j}}(\sqrt{\rho^j}\mathrm{U}^j)\cdot
    \nabla\mathrm{div}\phi_2d\mathrm{x}dt
    \rightarrow
    \int^T_0\int_{\mathbb{R}^N}\frac{h(\rho)}{\sqrt{\rho}}(\sqrt{\rho}\mathrm{U})\cdot
    \nabla\mathrm{div}\phi_2d\mathrm{x}dt,
    \end{equation}
as $j\rightarrow\infty$. From Corollary \ref{SSIC-C3.1}, Lemma
\ref{SSIC-L3.2} and Propositions \ref{SSIC-P3.3}-\ref{SSIC-P3.4},
we have
    \begin{eqnarray}
      &&  \int^T_0\int_{B_{\varepsilon_j,\frac{1}{k}}}(\sqrt{\rho^j}
      \mathrm{U}^j)\cdot(\nabla\bar{h}(\rho^j)\cdot\nabla)\phi_2d\mathrm{x}dt
      \nonumber\\
        &\leq&\int^T_0\int_{B_{\varepsilon_j,\frac{1}{k}}\cap\{|\mathrm{U}^j|\leq M\}}+
        \int_{B_{\varepsilon_j,\frac{1}{k}}\cap\{|\mathrm{U}^j|> M\}}
        (\sqrt{\rho^j}\mathrm{U}^j)\cdot(\nabla\bar{h}(\rho^j)\cdot\nabla)\phi_2d\mathrm{x}dt
      \nonumber\\
      &\leq&C\|\nabla\phi_2\|_{L^\infty_{tx}}\|\nabla{\bar{h}(\rho^j)}\|_{L^\infty(0,T;L^2)}
      \left(M\|\sqrt{\rho^j}\|_{L^{\frac{2\gamma(N+1)}{N}}(\mathbb{R}^N\times[0,T])}
      |B_{\frac{1}{k}}|^{\frac{\gamma(N+1)-N}{2\gamma(N+1)}}\right.\nonumber\\
        &&\left.+\frac{1}{1+\ln(1+M^2)}\|\rho^j|\mathrm{U}^j|^2(1+\ln(1+|\mathrm{U}^j|^2))\|_{L^\infty(0,T;L^1)}
      \right)\nonumber\\
     &&\rightarrow0,\ \mathrm{as}\ M,\
     k\rightarrow\infty,\nonumber
    \end{eqnarray}
    \begin{equation}
       \int^T_0\int_{B_{\frac{1}{k}}}(\sqrt{\rho}\mathrm{U})\cdot(\nabla\bar{h}(\rho)\cdot\nabla)\phi_2d\mathrm{x}dt
     \rightarrow0,\ \mathrm{as}\ k\rightarrow\infty,\nonumber
    \end{equation}
and
    \begin{equation}
    \int^T_0\int_{B_{\frac{1}{k},{n}}}(\sqrt{
    \rho^j}\mathrm{U}^j)\cdot(\nabla\bar{h}(\rho^j)\cdot\nabla)\phi_2d\mathrm{x}dt
    \rightarrow
    \int^T_0\int_{B_{\frac{1}{k},{n}}}(\sqrt{\rho}\mathrm{U})
    \cdot(\nabla\bar{h}(\rho)\cdot\nabla)\phi_2d\mathrm{x}dt
    \nonumber
    \end{equation}
as $j\rightarrow\infty$. Thus, we have
    \begin{equation}
    \int^T_0\int_{B_{\varepsilon_j,{n}}}(\sqrt{\rho^j}\mathrm{U}^j)\cdot
            (\nabla\bar{h}(\rho^j)\cdot\nabla)\phi_2d\mathrm{x}dt
            \rightarrow
    \int^T_0\int_{\mathbb{R}^N}(\sqrt{\rho}\mathrm{U})\cdot
            (\nabla\bar{h}(\rho)\cdot\nabla)\phi_2d\mathrm{x}dt,
    \end{equation}
as $j\rightarrow\infty$. Similarly, we can obtain
    \begin{equation}
    \int^T_0\int_{B_{\varepsilon_j,{n}}}(\sqrt{\rho^j}\mathrm{U}^j)\cdot
            (\nabla\phi_2\cdot\nabla)\bar{h}(\rho^j)d\mathrm{x}dt
            \rightarrow
    \int^T_0\int_{\mathbb{R}^N}(\sqrt{\rho}\mathrm{U})\cdot
            (\nabla\phi_2\cdot\nabla)\bar{h}(\rho)d\mathrm{x}dt,
    \label{SSIC-E3.35}
    \end{equation}
as $j\rightarrow\infty$.     From
(\ref{SSIC-E3.33})-(\ref{SSIC-E3.35}),
 we obtain
    $$
      \int^T_0\int_{|\mathrm{x}|> {\varepsilon_j}}2h(\rho^j)
      D(\mathrm{U}^j):\nabla\phi_2d\mathrm{x}dt
      \rightarrow <2h(\rho)D(\mathrm{U}),\nabla\phi_2>,\ \mathrm{as}\ j\rightarrow\infty.
    $$
Similarly, we obtain
     $$
      \int^t_0\int_{|\mathrm{x}|> {\varepsilon_j}}g(\rho^j)\mathrm{div
       U}^j\mathrm{div}
    \phi_2d\mathrm{x}dt
      \rightarrow <g(\rho)\mathrm{div}\mathrm{U},\mathrm{div}\phi_2>,\ \mathrm{as}\ j\rightarrow\infty.
    $$

Up to now, we have proved  that each term on the left hand side of
(\ref{SSIC-E3.28}) converges to corresponding term in
(\ref{SSIC-E2.11}) as $j\rightarrow\infty$. In the following, we
prove that  each term on the right hand side of (\ref{SSIC-E3.28})
vanishes as $j\rightarrow\infty$.

From Lemma \ref{SSIC-L3.1}, we get
        \begin{eqnarray}
          &&\left|\varepsilon_j \int^T_0\int_{\mathbb{R}^N}(\rho^j)^\theta\mathrm{div
       U}^j\mathrm{div}    \phi_2d\mathrm{x}dt
          \right|\nonumber\\
                &\leq&C\sqrt{\varepsilon_j}\|\nabla\phi_2\|_{L^\infty_{tx}}\left(
                \varepsilon_j\int^T_0\int_{\mathbb{R}^N}(\rho^j)^\theta|\nabla \mathrm{
       U}^j|^2d\mathrm{x}dt
                \right)^\frac{1}{2}\left(
                \int^T_0\int_{\mathbb{R}^N}\rho^jd\mathrm{x}dt
                \right)^\frac{\theta}{2}|B_{{n}}|^{\frac{1-\theta}{2}}\nonumber\\
                    &\leq&C(T)\sqrt{\varepsilon_j}n^{\frac{N(1-\theta)}{2}}\label{SSIC-E3.34}
        \end{eqnarray}
and
    \begin{equation}
    \left|\varepsilon_j \int^T_0\int_{\mathbb{R}^N}(\rho^j)^\theta
    D       (U^j):\nabla   \phi_2d\mathrm{x}dt
          \right|\leq C(T)\sqrt{\varepsilon_j}n^{\frac{N(1-\theta)}{2}}.\label{SSIC-E3.35-1}
    \end{equation}
It follows from (\ref{SSIC-E3.25}) and
(\ref{SSIC-E3.34})-(\ref{SSIC-E3.35-1}) that  each term on the
right hand side of (\ref{SSIC-E3.28}) vanishes as
$j\rightarrow\infty$.

Taking the limit $j\rightarrow\infty$ in (\ref{SSIC-E3.28}), we
finish the proof of this proposition.
\end{proof}

From the above arguments, we can immediately finish the proof of
Theorem \ref{SSIC-T2.1}.

\end{document}